\documentclass[11pt,]{amsart}
\usepackage{graphicx} 
\usepackage[utf8]{inputenc}
\usepackage{amsmath,amssymb}
\usepackage{wrapfig}
\usepackage{url}
\usepackage{mathtools}
\usepackage{graphicx}
\usepackage{stmaryrd}
\usepackage{amsthm}
\usepackage{xcolor}
\usepackage[shortlabels]{enumitem}
\usepackage{comment}
\usepackage{relsize}
\usepackage{ dsfont }
\usepackage{hyperref}
\usepackage{stmaryrd}

\usepackage{geometry}
\newtheorem{theorem}{Theorem}
\numberwithin{theorem}{section} 
\numberwithin{equation}{section}
\newtheorem{corollary}{Corollary}
\numberwithin{corollary}{section} 
\newtheorem{prop}{Proposition}
\numberwithin{prop}{section} 
\newtheorem{exemp}{Example}
\numberwithin{exemp}{section} 

\newtheorem{lemma}{Lemma}
\numberwithin{lemma}{section}
\theoremstyle{definition}
\newtheorem{obs}{Remark}
\theoremstyle{definition}
\newtheorem{definition}{Definition}
\numberwithin{definition}{section}

\DeclareMathOperator{\T}{\mathbb{T}}

\DeclareMathOperator{\G}{\mathnormal{\mathbb{T}^{m}\times\mathbb{R}^n}}
\DeclareMathOperator{\R}{\mathbb{R}}
\DeclareMathOperator{\Z}{\mathbb{Z}}
\DeclareMathOperator{\N}{\mathbb{N}}
\DeclareMathOperator{\s}{\mathcal{S}(\G)}

\DeclareMathOperator{\g}{\mathbb{T}^1\times\mathbb{R}}

\DeclareMathOperator{\Real}{\text{Re}}
\DeclareMathOperator{\Imag}{\text{Im}}

\DeclareFontFamily{U}{mathx}{\hyphenchar\font45}
\DeclareFontShape{U}{mathx}{m}{n}{
      <5> <6> <7> <8> <9> <10>
      <10.95> <12> <14.4> <17.28> <20.74> <24.88>
      mathx10
      }{}
\DeclareSymbolFont{mathx}{U}{mathx}{m}{n}
\DeclareFontSubstitution{U}{mathx}{m}{n}
\DeclareMathAccent{\widecheck}{0}{mathx}{"71}
\DeclareMathOperator{\ftil}{\widehat{\widehat{\mathnormal{f}\,}}\!\!\,}
\DeclareMathOperator{\util}{\widehat{\widehat{\mathnormal{u}\,}}\!\!\,}
\DeclareMathOperator{\vtil}{\widehat{\widehat{\mathnormal{v}\,}}\!\!\,}
\DeclareMathOperator{\Lvtil}{\widehat{\widehat{\mathnormal{Lv}\,}}\!\!\,}

\newcommand{\defeq}{\vcentcolon=}

\title{Fourier Analysis on $\T^m\times\R^n$ and applications to global hypoellipticity}
\author{André Pedroso Kowacs}


\begin{document}
\begin{abstract}
    This article presents a convenient approach to Fourier analysis for the investigation of functions and distributions on $\T^m \times \R^n$. Our approach involves the utilization of a mixed Fourier transform, incorporating both partial Fourier series on the torus on the first $m$ variables and partial Fourier transform in Euclidean space on the remaining variables. By examining the behaviour of the mixed Fourier coefficients, we achieve a comprehensive characterization of the spaces of rapidly decreasing smooth functions and distributions in this context. Additionally, we apply our results to derive necessary and sufficient conditions for the global hypoellipticity of a class first-order differential operators defined on $\T^1 \times \R$, including all constant coefficient first-order differential operators.   
\end{abstract}
\maketitle

\section{Introduction}
Fourier analysis has proven to be a powerful tool for studying partial differential equations not only in Euclidean space, but also in various smooth manifolds.  In particular, it has been extensively used in compact manifolds, such as the torus, in many recent papers as we may cite: \cite{Araujo2019},
\cite{Bergamasco1994}, \cite{Bergamasco1993}, \cite{Bergamasco2004},  \cite{Rafaeltoro}, \cite{bergamasco2015}, \cite{deAvilaSilva2018},  \cite{deAvilaSilva2018_2}, \cite{Gramchev1992}, \cite{Greenfield1972}, \cite{Greenfield1973}, \cite{Hounie1979}, \cite{Kirilov2019_2}, \cite{Kirilov2019}, 
\cite{KIRILOV2021108806}, \cite{KIRILOV2021699}. However, despite its wide-ranging applications, the study of differential operators on the product of the torus $\T^m = (\R/2\pi\Z)^m$ with $\R^n$, as studied by Avila and Cappielo in \cite{silva2025,DEAVILASILVA2022109418,silva2025solvable,silva2024systemsdifferentialoperatorstimeperiodic}, remains relatively unexplored.\\
\indent This paper focuses on developing a basic theory for Fourier analysis through partial Fourier transform on the the product space $\T^m\times\R^n$. Note that functions on this space can be identified with functions on $\R^{m+n}$ that are periodic on the first $m$ variables. We begin by establishing an appropriate space of test functions and distributions in this setting. By analyzing the behavior of their partial Fourier coefficients, we are able to characterize these functions and distributions, thereby enabling us to examine the global properties of differential operators on this product space.\\
\indent The motivation for this study stems from the works of F. Avila and M. Cappiello in \cite{DEAVILASILVA2022109418,silva2024systemsdifferentialoperatorstimeperiodic}, and Moraes, Kirilov and Ruzhansky in \cite{KIRILOV2020102853}, which have served as significant inspiration for our research. Building upon their contributions, we delve into the investigation of constant coefficient differential operators on $\T^1 \times \R$, aiming to determine necessary and sufficient conditions for these operators to preserve global regularity of solutions (Schwartz global hypoellipticity).

\section{Partial Fourier Analysis}
Let $m,n\in\N$. We denote by $\T^m\times\R^n$ the product of $\T^m = (\R/2\pi\Z)^m$, the $m$-dimensional torus and the $n$-dimensional Euclidean space $\R^n$. We also denote by $\mathcal{S}(\R^n)$ the Schwartz space on $\R^n$, and by $\mathcal{S}'(\R^n)$ its topological dual, also known as the space of tempered distributions. 
\begin{definition}\label{defipartialfourier}
    Let $f\in L^1(\G)$. Define its partial Fourier transform by
    $$\mathcal{F}_{\R^n}(f)(t,\xi)=\widehat{f}(t,\xi)\defeq \int_{\R^n}f(t,x)e^{-ix\cdot\xi}dx,
    $$
    for a.e. $(t,\xi)\in \G$, and
     $$\mathcal{F}_{\T^m}(f)(k,x)=\widehat{f}(k,x)\defeq \frac{1}{(2\pi)^m}\int_{[0,2\pi]^m}f(t,x)e^{-ik\cdot t}dt,
     $$
     for all $k\in\Z^m$ and a.e. $x\in\R^n$.
     Also define its mixed partial Fourier transform by
     \begin{align*}
         \mathcal{F}_{\T^m}(\mathcal{F}_{\R^n}(f))(k,\xi) = \ftil(k,\xi)&\defeq\frac{1}{(2\pi)^m}\int_{[0,2\pi]^m}\int_{\R^n}f(t,x)e^{-i(k\cdot t+x\cdot\xi)}dxdt,
     \end{align*}
     for all $k\in\Z^m$ and a.e. $\xi\in\R^n$.
     Note that by Fubini's Theorem, these are all well defined and moreover $\widehat{f}(\cdot,\xi)\in L^1(\T^m)$, for every $\xi\in\R^n$ and $\widehat{f}(k,\cdot),\ftil(k,\cdot)\in L^1(\R^n)$ for every $k\in\Z^m$. Furthermore, the following inequalities hold
     $$|\widehat{f}(t,\xi)|\leq \|f(t,\cdot)\|_{L^1(\R^n)},$$
     for almost every $t\in\T^m$, 
     $$|\widehat{f}(k,x)|\leq \frac{1}{(2\pi)^m}\|f(\cdot,x)\|_{L^1(\T^m)},$$
     for almost every $x\in\R^n$,
     $$|\ftil(k,\xi)|\leq \frac{1}{(2\pi)^m}\|f\|_{L^1(\T^m\times\R^n)},$$
     for almost every $(k,\xi)\in\Z^m\times\R^n$.
\end{definition}
In order to extend this definition to the appropriate set of distributions, we need to consider an appropriate space of test functions. While $\mathcal{S}(\R^n)$ and $C^\infty(\T^m)$ are adequate spaces of test functions in $\R^n$ and $\T^m$, respectively, for $\G$ we need a space of test functions which decay fast enough in $x$ uniformly in $t$. This motivates the following definition.
\begin{definition}
    Let $f\in C^\infty{(\G)}$. For each multi-index $\alpha\in\N^m_0,\,\beta,\gamma\in\N^n_0$, the real number
$$\|f\|_{\alpha,\beta,\gamma}\defeq\sup_{\substack{(t,x)\in\T^{m}}\times\R^n}|x^\gamma\partial_t^\alpha\partial_x^\beta f(t,x)|.$$
     Define $\mathcal{S}(\G)$ to be the set of all smooth functions on $\G$ which satisfy  $\|f\|_{\alpha,\beta,\gamma}<+\infty$ for every multi-index $\alpha\in\N_0^m,\,\beta,\gamma\in\N_0^n$
    with the topology induced by the countable family of (semi)norms:
    $$p_N(f)\defeq \sum_{|\alpha|+|\beta|+|\gamma|\leq N}\|f\|_{\alpha,\beta,\gamma},\qquad N\in\N_0.$$
    This provides $\mathcal{S}(\G)$ with a Fréchet space structure. With this in mind, we define $\mathcal{S}'(\G)$ to be its dual space, that is, the space of all linear functionals $u:\mathcal{S}(\G)\to\mathbb{C}$ such that there exist $N=N_u>0, C=C_u>0$ such that
    $$|\langle u,f\rangle|\leq Cp_N(f),$$
    for every $f\in \mathcal{S}(\G).$
\end{definition}
\begin{obs}
    Note that every continuous function $u:\T^m\times\R^n\to\mathbb{C}$ which satisfies
    $$|u(t,x)|\leq C(1+|x|^2)^{N/2},$$
    for some $C,N>0$ and all $(t,x)\in\G$ induces an element of $\mathcal{S}'(\G)$ by
    \begin{equation*}
        f\mapsto \int_{\G}u(t,x)f(t,x)dtdx.
    \end{equation*}
\end{obs}
\begin{obs}
    It is clear that $C^\infty(\T^m)\otimes\mathcal{S}(\R^n)\subset\s$, however, note that this inclusion is strict, that is
    $$\{f\in C^\infty(\G)| f(t,\cdot)\in\mathcal{S}(\R^n),\,\forall t\in \T^m\}\supsetneq\mathcal{S}(\G).$$
    To see why that is the case, consider that the function $f:\T^1\times\R\to\mathbb{C}$ given by
    \begin{align*}
    f(t,x) = \left\{
        \begin{alignedat}{3}
            &\tan(t)\exp\left(\frac{-1}{1-\left(x-\tan(t)\right)^2}\right),\quad&&\text{ for }|x-\tan(t)|<1,\,t\neq k\frac{\pi}{2},\,k\in\Z,\\
            &0,&&\text{ otherwise.}
        \end{alignedat}\right.
    \end{align*}
    Although smooth and compactly supported in $\R$, for each $t\in\T$, it is not in $\mathcal{S}(\G)$. Indeed, as 
 $\sup_{x\in\R}|f(t,x)| = \tan(t)e^{-1}$ for each $t\in(0,2\pi)\backslash\{\frac{\pi}{2},\frac{3\pi}{2}\}$, we have that $\sup_{t\in\T^1}\sup_{x\in\R}|f(t,x)|=+\infty$.
\end{obs}
Next we fully characterize $\mathcal{S}(\G)$ by the decay of the partial Fourier transforms of its elements.
\begin{prop}\label{proppartialdecayrn}
    A function $f\in C^\infty(\G)\cap L^1(\T^m\times\R^n)$ is in $\mathcal{S}(\G)$ if and only if for every $\beta,\in\N_0^n$, $N>0$, there exists $C_{N,\beta}>0$, such that
    \begin{equation}{\label{ineqpartialtorus}}
        |\partial_x^\beta\widehat{f}(k,x)|\leq C_{N,\beta}(1+|k|^2)^{-N/2}(1+|x|^2)^{-N/2},
    \end{equation}
    for every $k\in\Z^m,$ $x\in\R^n$.
\end{prop}
\begin{proof}
    First, suppose $f\in\mathcal{S}(\G)$. Note that integration by parts yields $\mathcal{F}_{\T^m}(\partial_{t_j}f)(k,x) = (ik_j)\widehat{f}(k,x)$, for every $j=1,\dots,m$. Let $\beta,$ and $N$ be as stated. Then
    $$\mathcal{F}_{\T^m}\left((1-\Delta_t)^{N/2}(1+|x|^2)^{N/2}\partial_x^\beta f\right)(k,x) = (1+|k|^2)^{N/2}(1+|x|^2)^{N/2}\partial_x^\beta \widehat{f}(k,x),$$
    so 
    \begin{align*}
        |(1+|k|^2)^{N/2}(1+|x|^2)^{N/2}&\partial_x^\beta \widehat{f}(k,x)|\leq \\
        &\leq\frac{1}{(2\pi)^m}\int_{\T^m}|\left((1-\Delta_t)^{N/2}(1+|x|^2)^{N/2}\partial_x^\beta f\right)(t,x)e^{-ik\cdot t}|dt\\
        &\leq \sup_{\substack{t\in\T^m\\x\in\R^n}}|\left((1-\Delta_t)^{N/2}(1+|x|^2)^{N/2}\partial_x^\beta f\right)(t,x)|\\
        &\leq {C}_{N,\beta}<+\infty,
    \end{align*}
 for some ${C}_{N,\beta}>0$ and every $k\in\Z^m,$ $x\in\R^n$, and so inequality (\ref{ineqpartialtorus}) follows. Conversely, suppose (\ref{ineqpartialtorus}) holds. Then these inequalities imply that the series
    $$g(t,x) \defeq \sum_{k\in\Z^m}\widehat{f}(k,x)e^{ik\cdot t}$$ converges absolutely and uniformly in $x$. Moreover, for each multi-index $\alpha\in\N_0^m$, $\beta,\gamma\in\N_0^n$, the series
    $$\sum_{k\in\Z^m}(ik)^\alpha \partial_x^\beta\widehat{f}(k,x)e^{ik\cdot t}$$
    converges absolutely and uniformly in $x$, therefore $g\in C^\infty(\G)$ and $\partial_t^\alpha \partial_x^\beta {g}(t,x)$ is given by the series above. Moreover, the usual Fourier inversion theorem implies that $g=f$ point-wise in $x$. Also, note that for every $\alpha\in\N_0^m$, $\beta,\gamma\in\N_0^n$ we have that
    \begin{align*}
        |x^{\gamma}\partial_t^\alpha\partial_x^\beta{f}(t,x)|&\leq \sum_{k\in\Z^m} |k^\alpha x^{\gamma}\partial_x^\beta\widehat{f}(k,x)|\\
        &\leq \sum_{k\in\Z^m} |k|^{|\alpha|} C_{|\alpha|+\gamma+m+1,\beta}(1+|k|^2)^{-(|\alpha|+m+1)/2}|x|^{\gamma}(1+|x|^2)^{-(|\alpha|+m+1)/2}\\
        &\leq\tilde{C}_{\alpha,\beta,\gamma}<+\infty,
    \end{align*}
 for some $\tilde{C}_{\alpha,\beta,\gamma}$, for each $(t,x)\in\G$, so that $f\in\mathcal{S}(\G)$ as claimed.
\end{proof}

With this result in mind, given a sequence of functions $\{{g}(k,\cdot)\in C^\infty(\R^n)\}_{k\in\Z^m}$ which satisfies inequalities (\ref{ineqpartialtorus}), we define its inverse partial Fourier transform by
$$\mathcal{F}^{-1}_{\T^m}(g)(t,x) = \widecheck{g}(t,x)\defeq  \sum_{k\in\Z^m}g(k,x)e^{ik\cdot t},\,(t,x)\in\G.$$
This way, from the proof of last proposition, we have that
$$\mathcal{F}^{-1}_{\T^m}(\mathcal{F}_{\T^m}(f))(t,x)=f(t,x),\,\forall f\in \mathcal{S}(\G).$$

\begin{prop}\label{propdecaypartialxi}
    A function $f\in C^\infty(\G)\cap L^1(\T^m\times\R^n)$ is in $\mathcal{S}(\G)$ if and only if for every $\alpha\in\N_0^m,\,\beta\in\N_0^n$, $N>0$, there exists $C_{N,\alpha,\beta}>0$, such that
    \begin{equation}\label{ineqpartialr}
|\partial_t^\alpha\partial_\xi^\beta\widehat{f}(t,\xi)|\leq C_{N,\alpha,\beta}(1+|\xi|^2)^{-N/2},
    \end{equation}
for every $t\in\T^m$, $\xi\in\R^n$.
\end{prop}
\begin{proof}
     First, suppose $f\in\mathcal{S}(\G)$. Note that integration by parts and Lebesgue's dominated convergence theorem yield $\mathcal{F}_{\R^n}(\partial_{x_j}[(-ix_l)f])(t,\xi) = (i\xi_j)\partial_{\xi_l}\widehat{f}(t,\xi)$, for every $j,l=1,\dots,n$. Now let $\alpha,\beta$ and $N$ be as stated. Then
      $$\mathcal{F}_{\R^n}\left((1-\Delta_x)^{N/2}[(-ix)^\beta \partial_t^\alpha f]\right)(t,\xi) = (1+|\xi|^2)^{N/2}\partial_\xi^\beta\partial_t^\alpha\widehat{f}(t,\xi),$$
      so 
    \begin{align*}
        |(1+|\xi|^2)^{N/2}\partial_{\xi}^\beta\partial_t^\alpha\widehat{f}(t,\xi)|&\leq \int_{\R^n}\left|(1-\Delta_x)^{N/2}[(-ix)^\beta\partial_t^\alpha f(t,x)]e^{-ix\cdot \xi}\right|dx\\
        &\leq \sup_{\substack{t\in\T^m\\x\in\R^n}}\left|(1+|x|^2)^{\frac{n+1}{2}}(1-\Delta_x)^{N/2}[(-ix)^\beta \partial_t^\alpha f(t,x)]\right|\\
        &\times\int_{\R^n}(1+|x|^2)^{-\frac{n+1}{2}}dx\\
        &\leq C_{N,\alpha,\beta}<+\infty,
    \end{align*}
    for some $C_{N,\alpha,\beta}>0$ and for every $t\in\T^m,\,\xi\in\R^n$,
    and inequality (\ref{ineqpartialr}) follows. Conversely, suppose inequality (\ref{ineqpartialr}) holds. This implies that for each $\alpha\in\N_0^m,\,\beta,\gamma\in\N_0^n$ there exists $ C_{\alpha,\beta,\gamma}>0$,
    such that $$|\xi^\gamma\partial_{\xi}^\beta\partial_t^\alpha\widehat{f}(t,\xi)e^{ix\cdot\xi}|\leq C_{\alpha,\beta,\gamma}(1+|\xi|^2)^{-\frac{n+1}{2}},$$
    for every $t\in\T^m$, $\xi\in\R^n$. Therefore, by Lebesgue's dominated convergence theorem, it follows that the function
    $$h(t,x) = \frac{1}{(2\pi)^n}\int_{\R^n}\widehat{f}(t,\xi)e^{ix\cdot\xi}d\xi$$
    is smooth on $\G$. Moreover, by the usual Fourier inversion theorem, we have that $h=f$ point-wise in $t$. Finally, notice that for every $\alpha\in\N_0^m,\,\beta,\gamma\in\N_0^n$, we have that
    \begin{align*}
        |x^\gamma\partial_t^\alpha\partial_x^\beta{f}(t,x)|&\leq \frac{1}{(2\pi)^n}\int_{\R^n}|\partial_\xi^\gamma\xi^\beta\partial_t^\alpha\widehat{f}(t,\xi)|d\xi\\
        &\leq \frac{1}{(2\pi)^n} \int_{\R^n}|\xi|^\beta \tilde{C}_{\alpha,\beta,\gamma}(1+|\xi|^2)^{-(|\beta|+n+1)/2}d\xi\\
        &\leq \Tilde{C}'_{\alpha,\beta,\gamma}<+\infty,
    \end{align*}
    for some $\tilde{C}_{\alpha,\beta,\gamma},\tilde{C}'_{\alpha,\beta,\gamma}>0$, for each $(t,x)\in\G$, so that $f\in\mathcal{S}(\G)$.
\end{proof}
\begin{corollary}\label{corofourierschwartz}
    Let $f\in C^\infty(\G)$. Then $f\in\mathcal{S}(\T^m\times\R^n)\iff \mathcal{F}_{\R^{n}}(f)\in\mathcal{S}(\G)$.
\end{corollary}
With this result in mind, for $g\in \mathcal{S}(\G)$, we define its inverse partial Fourier transform by
$$\mathcal{F}^{-1}_{\R^n}(g)(t,x) = \widecheck{g}(t,x)\defeq  \int_{\R^n} g(t,\xi)e^{ix\cdot \xi}d\xi,\,(t,x)\in\G.$$
This way, from the proof of last proposition, we have that: 
$$\mathcal{F}^{-1}_{\R^n}(\mathcal{F}_{\R^n}(f))(t,x)=f(t,x),\,\forall f\in \mathcal{S}(\G).$$
\begin{prop}\label{proppartialdecaymix}
    Let $f\in C^\infty(\G)\cap L^1(\G)$. Then $f\in \mathcal{S}(\G)$ if and only if for each $\beta\in\N_0^n$, $N>0$ there exists $C_{N,\beta}>0$ such that
    \begin{equation}\label{ineqmix}
|\xi^\gamma\partial_{\xi}^\beta\ftil(k,\xi)|\leq C(1+|k|^2)^{-N/2}(1+|\xi|^2)^{-N/2},
    \end{equation}
    for every $k\in\Z^m$, $\xi\in\R^n$.
\end{prop}
\begin{proof}
    Suppose first $f\in\mathcal{S}(\G)$. Notice that:
    \begin{align*}
            &\mathcal{F}_{\R^n}(\mathcal{F}_{\T^m}((1-\Delta_x)^{N/2}((-ix)^\beta(1-\Delta_t)^{N/2}f))))(k,\xi) =\\
            &=(1+|\xi|^2)^{N/2}(1+|k|^2)^{N/2}\partial_{\xi}^\beta\ftil(k,\xi).
    \end{align*}
    Therefore
    \begin{align*}
        &(1+|\xi|^2)^{N/2}(1+|k|^2)^{N/2}|\xi^\gamma\partial_{\xi}^\beta\ftil(k,\xi)|= \\
        &=\frac{1}{(2\pi)^m}\left|\int_{\R^n}\int_{\T^m}(1-\Delta_x)^{N/2}((-ix)^\beta(1-\Delta_t)^{N/2}f(t,x)e^{-i(k\cdot t+x\cdot\xi)}dtdx\right|\\
        &\leq\frac{1}{(2\pi)^m}\sup_{\substack{t\in\T^m\\x\in\R^n}}|(1+|x|^2)^{\frac{n+1}{2}}(1-\Delta_x)^{N/2}\partial_x^\gamma((-ix)^\beta(1-\Delta_t)^{N/2}f(t,x)|\\
        &\times\int_{\R^n}\int_{\T^m}(1+|x|^2)^{-\frac{n+1}{2}}dtdx<+\infty,
    \end{align*}
    for every $k\in\Z^m$, $\xi\in\R^n$,
    so inequality (\ref{ineqmix}) follows. Conversely, suppose (\ref{ineqmix}) holds. 
     This means that for each $\alpha\in\N_0^m$, $\beta,\gamma\in\N_0^n$ there exists $ C_{\alpha,\beta,\gamma}>0$ such that
    $$|k^\alpha\xi^\beta\partial_{\xi}^\gamma\ftil(k,\xi)|\leq C_{\alpha,\beta,\gamma}(1+|k|^2)^{-\frac{m+n+1}{2}}(1+|\xi|^2)^{-\frac{m+n+1}{2}},$$
    for every $k\in\Z^m$, $\xi\in\R^n$. Therefore, by Lebesgue's dominated convergence theorem, it follows that the function
    $$h_\alpha(k,x) = \frac{1}{(2\pi)^n}\int_{\R^n}(ik)^\alpha\ftil(k,\xi)e^{ix\cdot\xi}d\xi$$
    is smooth in $x\in\R^n$, for each $k\in\Z^m$. Moreover, $\exists K_\alpha>0$ such that
    $$|h_\alpha(k,x)|\leq K_\alpha(1+|k|^2)^{-\frac{m+n+1}{2}},$$
    for each $x\in\R^n.$ This means that the function
    $$g(t,x) = \sum_{k\in\Z^m}h_0(k,x)e^{ik\cdot t} = \sum_{k\in\Z^m}\frac{1}{(2\pi)^n}\int_{\R^n}\ftil(k,\xi)e^{ix\cdot\xi}d\xi e^{ik\cdot t}$$
    is smooth for $(t,x)\in\G$. Finally, we see point-wise in $k$ and then in $x$, that by the usual Fourier inversion theorem $\mathcal{F}_{\R^n}(h_0)(k,\xi) = \ftil(k,\xi)$ and $\mathcal{F}_{\T^m}(g)(k,x) = h_0(k,x)$, so that $f = g$.
\end{proof}
    We then extend the definition of the partial Fourier transforms to a more general setting, by considering the ``tempered" distributions of $\G$.
\begin{definition}
    Let $u\in \mathcal{S}'(\G)$. We define its partial Fourier transforms by
    $$\langle\mathcal{F}_{\R^n}(u)(t,\xi),f(t,\xi)\rangle = \langle u(t,\xi),\mathcal{F}_{\R^n}(f)(t,\xi)\rangle,$$
    for every $ \,f\in\mathcal{S}(\G)$, and
    $$\mathcal{F}_{\T^m}(u)(k,\cdot)\in \mathcal{S}'(\R^n),$$
   acting by 
   $$\langle\mathcal{F}_{\T^m}(u)(k,x),f(x)\rangle = \frac{1}{(2\pi)^m}\langle u(t,x),e^{-ik\cdot t}f(x)\rangle,$$
   for every $f\in\mathcal{S}(\R^n)$.
    We also define its mixed partial Fourier transform by
    $$\mathcal{F}_{\T^m}(\mathcal{F}_{\R^n}(u))(k,\cdot)\in \mathcal{S}'(\R^n),$$
    acting by
    $$\left\langle \mathcal{F}_{\T^m}(\mathcal{F}_{\R^n}(u))(k,\xi),f(\xi)\right\rangle = \frac{1}{(2\pi)^m}\langle u(t,\xi),e^{-ik\cdot t}\mathcal{F}_{\R^n}(f)(\xi)\rangle,$$
    for every $f\in\mathcal{S}(\R^n)$.
    For $u\in\mathcal{S}'(\G)$  we also define $$\langle\mathcal{F}_{\R^n}^{-1}u(t,x),f(t,x)\rangle = \langle u(t,x),\mathcal{F}_{\R^n}^{-1}f(t,x)\rangle.$$
    Notice that for $u\in\mathcal{S}'(\T^m\times\R^n)$ we have that $\mathcal{F}_{\R^n}u,\mathcal{F}_{\R^n}^{-1}u\in\mathcal{S}'(\G)$. We will also sometimes denote the partial Fourier transforms of $u\in\mathcal{S}'(\T^m\times\R^n)$ by $\widehat{u}$.
\end{definition}
Note that by Corollary \ref{corofourierschwartz} and the inclusion $C^\infty(\T^m)\otimes\mathcal{S}(\R^n)\subset\mathcal{S}(\G)$, these distributions are all well defined. Additionally, it is easy to see that this definition is consistent with Definition \ref{defipartialfourier}.
\begin{prop}\label{proppartialdecaytorusdistrib}
    Let $u\in \mathcal{S}'(\G)$. Then there exist $C,N>0$ such that
    \begin{equation}\label{ineqpartialtorusdistrib}
    |\langle \widehat{u}(k,x),f(x)\rangle|\leq C\Tilde{p}_N(f)(1+|k|^2)^{N/2},
    \end{equation}
    where $\Tilde{p}_N(f) = \sum_{|\alpha|+|\gamma|\leq N}\|x^\gamma\partial^\alpha f\|_\infty$, for every $f\in \mathcal{S}(\R^n)$, $k\in\Z^m$.
\end{prop}
\begin{proof}
Let $f\in\mathcal{S}(\R^n)$. Then from definition, there exist $C',N>0$ such that
    \begin{align*}
        |\langle \widehat{u}(k,x),f(x)\rangle| &= |\langle u,e^{ik\cdot t}f(x)\rangle|\leq C' p_N(e^{ik\cdot t}f(x))\\
        &\leq C'C_N(1+|k|^2)^{N/2}\Tilde{p}_N(f),
    \end{align*}
    for some $C>0$ and for every $k\in\Z^m$.
\end{proof}
\begin{lemma}\label{lemmacauchys'}
    Let $(u_j)_{j\in\N}\subset\mathcal{S}'(\G)$ such that, for each $\theta\in\mathcal{S}(\G)$, $(\langle u_j,\theta\rangle)_{j\in\N}$ is a Cauchy sequence in $\mathbb{C}$. Then there exists a unique $u\in\mathcal{S}'(\G)$ such that $u=\lim_{j\to\infty}u_j$.
\end{lemma}
\begin{proof}
    The proof is similar to the proof in [Theorem 2.9 in \cite{DEAVILASILVA2022109418}]. Define
    $$\langle u ,\phi\rangle = \lim_{j\to\infty}\langle u_j,\phi\rangle\in\mathbb{C},$$
    for each $\phi\in\mathcal{S}(\G)$. Clearly,, the mapping $\phi\mapsto\langle u,\phi\rangle$ islinear; therefore, itonly remains to prove that it is continuous. Suppose $\phi_l\to0$ in $\s$ as $l\to\infty$. Since for each $\psi\in\s$, the sequence $(\langle u_j,\psi\rangle)_{j\in\N}$ is a Cauchy sequence in $\mathbb{C}$, it is bounded, say by $0<C_\psi<+\infty$. As $\s$ is a Fréchet space and $\mathbb{C}$ is a Banach space, $\{u_j\}_{j\in\N}$ is equicontinuous on the unit ball, by the uniform boundness principle. Therefore, there exist $K,N>0$ such that
    $$|\langle u_j,\psi\rangle|\leq K,$$
    for each $\psi\in\s$  satisfying $p_N(\psi)\leq 1$. Let $\epsilon>0$ and set 
    $$\psi_l=\frac{2K}{\epsilon}\phi_l,\,l\in\N.$$
    Then $\psi_l\to0$ in $\s$ as well, and so $\exists l_0\in\N$ such that $p_N(\psi_l)\leq 1$ for all $l\geq l_0$. Therefore, for each $l\geq l_0$ we have that 
    $$|\langle u_j,\phi_l\rangle|\leq\frac{\epsilon}{2},$$
    for each $j\in\N$. Also, since $\langle u,\phi_l\rangle = \lim_{j\to\infty}\langle u_j,\phi_l\rangle$, for each $l\geq l_0$, there exists $j_l\in\N$ such that
    $$|\langle u,\phi_l\rangle-\langle u_{j_l},\phi_l\rangle|<\frac{\epsilon}{2}.$$
    Therefore, we conclude that for $l\geq l_0$ we have that
    $$|\langle u,\phi_l\rangle| \leq |\langle u,\phi_l\rangle-\langle u_{j_l},\phi_l\rangle| + |\langle u_{j_l},\phi_l\rangle|< \epsilon,$$
    so that $\langle u,\phi_l\rangle\to 0$ and $u\in\mathcal{S}'(\G)$.
\end{proof}
\begin{prop}\label{proppartialdecaydistribinv}
    Let $\{{u}(k,\cdot)\in \mathcal{S}'(\R^n)\}_{k\in\Z^m}$ be a sequence of tempered distributions which satisfy inequalities (\ref{ineqpartialtorusdistrib}). Then there exists a unique $\widecheck{u}\in \mathcal{S}'(\G)$ such that
    $$\langle \widecheck{u} , f\rangle = (2\pi)^m\sum_{k\in\Z^m}\langle {u}(k,x),\widehat{f}(-k,x)\rangle=\sum_{k\in\Z^m}\langle{u}(k,x)e^{ik\cdot t},f(t,x)\rangle,$$
    so that we define
    $\mathcal{F}^{-1}_{\T^m}({u}) \defeq \widecheck{{u}}.$
\end{prop}
\begin{proof}
For each $j\in\N$, define $S_j\in \mathcal{S}'(\G)$ by
\begin{align*}
\langle S_j(t,x),f(t,x)\rangle\ &= \sum_{|k|\leq j}\langle \widehat u(k,x)e^{ik\cdot t},f(t,x)\rangle = \sum_{|k|\leq j}\int_{\T^m}\langle \widehat u(k,x),f(t,x)\rangle e^{ik\cdot t}dt  \\
&=  (2\pi)^m\sum_{|k|\leq j}\langle \widehat u(k,x),\widehat{f}(-k,x)\rangle,
\end{align*}
for every $f\in\s$.
Note that for each $k\in\Z^m$, inequality (\ref{ineqpartialtorusdistrib}) implies that 
\begin{align*}
    |\langle {u}(k,x),\widehat{f}(-k,x)\rangle|&\leq \Tilde{C}\Tilde{p}_{N}\left(\widehat{f}(-k,\cdot)\right)(1+|k|^2)^{N/2}
\end{align*}
for some $\Tilde{C}$. 
Therefore, for every $j,l\in\N$ we have that
\begin{align*}
    |\langle S_{j+l},f \rangle - \langle S_{j},f\rangle|&\leq \sum_{|k|>j}^{j+l} |\langle {u}(k,x),\widehat{f}(-k,x)\rangle|\leq \sum_{|k|>j}^{j+l} \Tilde{C}\Tilde{p}_{N}\left(\widehat{f}(-k,\cdot)\right)(1+|k|^2)^{N/2}\\
    &\leq\sum_{|k|>j}^{\infty} \Tilde{C}\Tilde{p}_{N}\left(\widehat{f}(-k,\cdot)\right)(1+|k|^2)^{N/2}.
\end{align*}
As $f\in\s$, there exists $C_f>0$ such that
$$\Tilde{p}_{N}\left(\widehat{f}(-k,\cdot)\right)\leq C_f(1+|k|^2)^{-\frac{N+m+1}{2}},$$
for every $k\in\Z^m$, by the same argument used in the proof of Proposition \ref{proppartialdecayrn}. This implies 
$$|\langle S_{j+l},f \rangle - \langle S_{j},f\rangle|\leq \Tilde{C}C_f\left(\sum_{\substack{|k|>j\\k\in\Z^m}}^{\infty}(1+|k|^2)^{-\frac{m+1}{2}}\right)\to 0$$
as $j\to+\infty$, as the series above is convergent for each $j\in\N$. This means that the sequence $(\langle S_j,f\rangle)_{j\in\N}$ is a Cauchy sequence in $\mathbb{C}$, for each $f\in\s$, so that it follows from Lemma \ref{lemmacauchys'} that 
\begin{align*}
    \widecheck{u} = (2\pi)^m\sum_{k\in\Z^m}{u}(k,x)e^{ik\cdot t}=\lim_{j\to\infty}S_j\in\mathcal{S}'(\G).
\end{align*}
\end{proof}
Note using both Propositions \ref{proppartialdecaytorusdistrib} and \ref{proppartialdecaydistribinv}, for each $u\in \mathcal{S}'(\G)$ we have that
$$\mathcal{F}^{-1}_{\T^m}(\mathcal{F}^{}_{\T^m}(u)) = u.$$
\begin{prop}\label{proppartialdecayrndistribinv}
     Let $u:\G\to\mathbb{C}$ be such that there exist $C,N>0$ such that
\begin{equation}\label{ineqpartialrndistrib}
    |{u}(t,\xi)|\leq C(1+|\xi|^2)^{N/2}, 
    \end{equation}
     for each $(t,\xi)\in\G$. Then there exists a unique $\widecheck{u}\in\mathcal{S}'(\G)$ given by
     \begin{align*}
         \langle \widecheck{u},f\rangle &=\frac{1}{(2\pi)^n}\int_{\T^m} \int_{\R^n}\int_{\R^n}u(t,\xi)e^{ix\cdot\xi}f(t,x)d\xi dxdt\\
         &=\frac{1}{(2\pi)^n}\int_{\T^m} \int_{\R^n} u(t,\xi)\widehat{f}(t,-\xi)d\xi dt,
      \end{align*}
      for each $f\in \mathcal{S}(\G)$ and we define $\mathcal{F}^{-1}_{\R^n}({u}) \defeq \widecheck{u}.$
\end{prop}
\begin{proof}
 Let $u$ be as stated. Then inequality (\ref{ineqpartialrndistrib}) implies the integrals above are well defined and finite, so it is clear that the map $f\mapsto\langle\widecheck{u},f\rangle$ is linear. It only remains to prove that it is continuous. Hence, let $f_j\to 0$ in $\s$. Then $p_M(f_j)\to 0$ for each $M>0$. But then Proposition \ref{proppartialdecayrn} and its proof imply that $p_M(\widehat{f_j})\to 0$ for each $M>0$. This means that
 \begin{align*}
     |\langle \widecheck{u},f_j\rangle|&\leq \frac{1}{(2\pi)^n}\int_{\T^m}\int_{\R^n}|u(t,\xi)|(1+|\xi|^2)^{-\frac{N+n+1}{2}}(1+|\xi|^2)^{\frac{N+n+1}{2}}|\widehat{f_j}(t,-\xi)|d\xi dt\\
&\leq(2\pi)^{m-n}p_{N+n+1}\left(\widehat{f_j}\right)C\int_{\R^n}(1+|\xi|^2)^{-\frac{n+1}{2}}d\xi\to0
 \end{align*}
 as $j\to+\infty$, for some $C>0$. It follows that $\widecheck{u}\in\mathcal{S}'(\G)$.
\end{proof}
Notice that for $u$ as in the previous proposition, we have that
$$\mathcal{F}^{}_{\R^n}(\mathcal{F}^{-1}_{\R^n}(u)) = u$$
in the sense of distributions.
\begin{corollary}\label{corodistribmixinv}
     Let $\{{u}(k,\cdot):\R^n\to\mathbb{{C}}\}_{k\in\Z^m}$ be a sequence of functions such that $\exists N,C>0$, such that:
     \begin{equation}
         |u(k,\xi)|\leq C(1+|k|^2)^{N/2}(1+|\xi|^2)^{N/2}
     \end{equation}
     for all $(k,\xi)\in\Z^m\times\R^n$. Then $\mathcal{F}^{-1}_{\R^n}(\mathcal{F}^{-1}_{\T^m}(u))$ is well defined and belongs to $\mathcal{S}'(\G)$. Furthermore, $\mathcal{F}^{}_{\T^m}(\mathcal{F}^{}_{\R^n}(\mathcal{F}^{-1}_{\R^n}(\mathcal{F}^{-1}_{\T^m}(u)))) = u$ in the sense of distributions.
\end{corollary}
\begin{proof}
    Note that if $u$ is as stated, then there exist $C',N'>0$ such that: 
    $$\left|\frac{1}{(2\pi)^n}\int_{\T^m} \int_{\R^n} u(k,\xi)\widehat{f}(-\xi)d\xi dt\right|\leq C'\Tilde{p}_{N'}(f)(1+|k|^2)^{N'/2}$$
    for each $f\in\mathcal{S}(\R^n)$, so that the claim follows from  Propositions \ref{proppartialdecaydistribinv} and \ref{proppartialdecayrndistribinv}.
\end{proof}

It is worth noting that as in $\mathcal{S}(\R^n)$, for $u\in\mathcal{S}'(\G)$, we may define as usual $gu$ and $\partial_t^{\alpha}\partial_x^{\beta}u$, where $g:\T^m\times\R^n\to\mathbb{C}$ is a suitable function (as with polynomial growth in $\R^n$, uniform in $t\in\T^m$, along with all of its derivatives) and $\alpha\in\N^m_0,\,\beta\in\N^n_0$.  It is also easy to see that in $\mathcal{S}'(\g)$ the following equalities also hold
\begin{align*}
    \mathcal{F}_{\T^m}(\partial_t^\alpha u)(k,x) = (ik)^\alpha \widehat{u}(k,x),\\
    \mathcal{F}_{\R^n}(\partial_x^\beta u)(k,x) = (i\xi)^\beta \widehat{u}(t,\xi),\\
    \mathcal{F}_{\R^n}((-ix)^\gamma u)(t,\xi) =  \partial_{\xi}^\gamma\widehat{u}(t,\xi).
\end{align*}
\section{Application on regularity of Differential Operators}
The previous results allow the study of regularity differential operators on $\G$ as follows.
\begin{definition}
    Let $L:\mathcal{S}'(\G)\to\mathcal{S}'(\G)$ be a differential operator. We say $L$ is Schwartz globally hypoelliptic (SGH) if whenever $Lu=f\in\s$ for some $u\in\mathcal{S}'(\G)$, this implies $u\in\s$.
\end{definition}
As an example, consider the differential operator $L=\partial_t+\partial_x$ on $\T^1\times\R$. Then clearly $L$ is not SGH since $v(t,x)\equiv1\in\mathcal{S}'(\G)\backslash\s$ satisifies $Lv=0\in\s$.
\subsection{Constant Coefficient case}
The next theorem provides necessary and sufficient conditions for a class of constant coefficient differential operator acting on on $\T^1\times\R$ to be SGH.

\begin{theorem}
    Let $L$ be a constant coefficient differential operator on $\T^1\times\R$, given by
    \begin{equation}
        L=p(\partial_x)+q(\partial_t),
    \end{equation}
    where $p$ and $q$ are non-constant polynomials in one variable, and $p$ or $q$ have real coefficients. Then $L$ is SGH if and only if 
    \begin{equation}\label{setZ}
        Z=\left\{p(\xi)+q(k)=0|(k,\xi)\in\Z\times\R\right\}=\emptyset.
    \end{equation}
\end{theorem}
\begin{proof}
    First suppose there exists $(k_0,\xi_0)\in Z$. Then the function 
    $$v(t,x) = \frac{1}{2\pi}e^{i(k_0t+\xi_0x)}$$
    satisfies $v\in\mathcal{S}'(\T^1\times\R)\backslash\mathcal{S}(\T^1\times\R)$ and $Lv=0\in\mathcal{S}(\T^1\times\R)$, so that $L$ is not SGH.\\
    Suppose now $Z$ is empty. We will assume that $q$ has real coefficients, the proof of the other case is analogous. 
    We claim that there exists $\varepsilon>0$ such that $|p(\xi)+q(k)|\geq \varepsilon$, for every $(k,\xi)\in\Z\times\R$. Indeed, if $Z$ is empty, then either $\Real(p(\xi))\neq 0$, for all $\xi\in\R$, or whenever $\Real(p(\xi))=0$, we have that $\Imag(p(\xi))+q(k)\neq0$. If $\Real(p(\xi))\neq 0$, since polynomials in one variable are closed maps, we have that there exists $\epsilon>0$ such that $\Real(p(\xi))\geq \epsilon$ and so
    \begin{equation*}
        |p(\xi)+q(k)|\geq |\Real(p(\xi))|\geq \epsilon.
    \end{equation*}
    On the other hand, if $\Real(p(\xi))=0$ for some $\xi\in\R$, then by the fundamental theorem of algebra, there are at most $\deg(p)$ values for $\xi$ where this occurs, say at $\xi_1,\dots,\xi_n$. Since $\Imag(p(\xi_j))+q(k)\neq 0$, for $j=1,\dots,n$, and $q$ is also a closed map, we have that for each $j=1,\dots,n$ there exists $\epsilon_j$ such that 
    \begin{equation}\label{closed}
        |\Imag(p(\xi_j))+q(k)|\geq \epsilon_j,
    \end{equation}
    for every $k\in\Z$ and for every $j=1,\dots,n$. Take $\epsilon=\min\{\epsilon_j,\,j=1,\dots,n\}$. Since $\Imag(p(\xi))$ is continuous, there exists $\delta>0$ such that whenever $|\xi-\xi_j|<\delta$ we have that
    \begin{equation}\label{close}
        |p(\xi)-p(\xi_j)|<\epsilon/2.
    \end{equation}
    Moreover, since $\Real(p(\xi))$ is a closed map, we have that there exists $\epsilon'>0$ such that 
    \begin{equation}\label{faraway}
        |\Real(p(\xi)|>\epsilon', 
    \end{equation}
    whenever $\xi\in\R\backslash\cup_{j=1}^n(\xi_j-\delta,\xi_j+\delta)$.
    Therefore, if $|\xi-\xi_j|<\delta$ we have that
    \begin{align*}
        |p(\xi)+q(k)|&\geq \frac{1}{2}(|\Real(p(\xi))|+(|\Imag(p(\xi_j)+q(k)|-|\Imag(p(\xi_j))-\Imag(p(\xi))|))\\
        &\geq \frac{1}{2}(\epsilon-\epsilon/2)=\epsilon/4,
    \end{align*}
    for every $k\in\Z$, by \eqref{closed} and \eqref{close}. On the other hand, if $|\xi-\xi_j|\geq \delta$, for every $j=1,\dots,n$, we have that
    \begin{equation*}
         |p(\xi)+q(k)|\geq |\Real(p(\xi))|>\epsilon',
    \end{equation*}
    for every $k\in\Z$, by \eqref{faraway}. Taking $\varepsilon=\min\{\epsilon/4,\epsilon'\}$ yields the claim. Now if $Lu=f\in\mathcal{S}(\g)$, then since $p(\xi)+q(k)$ never vanishes we have that
    \begin{equation}\label{equtilf}
    \util(k,\xi) = \frac{\ftil(k,\xi)}{i(p(\xi)+q(k))}.
    \end{equation}
    Therefore by the Leibniz rule for differentiation, and the fact that $|p(\xi)+q(k)|\geq \varepsilon$ for every $(k,\xi)\in\Z\times\R$, for every $\gamma,\beta\in\N$ we have that
$$|\partial_{\xi}^\beta\util(k,\xi)|\leq K_{\beta}\max_{\gamma\leq m-1,\,\beta'\leq \beta}|\xi^{\gamma}\partial_{\xi}^{\beta'}\ftil(k,\xi)|,$$ for some $K_{\beta}>0$ and  $m=\deg(p)$, so that $f\in\mathcal{S}(\T^1\times\R)\implies u\in\mathcal{S}(\T^1\times\R)$ and $L$ is SGH.\\
\end{proof}

\begin{corollary}\label{coro_cte_cases}

    Let $L$ be a first order constant coefficients linear differential operator on $\T^1\times\R$, so that
    $$L = c_1\partial_t+c_2\partial_x+c_3,$$
for some $c_1,c_2,c_3\in\mathbb{C}$, $c_1\neq0$ or $c_2\neq 0$. Then $L$ is Schwartz globally hypoelliptic if and only if
$$Z\defeq \{(k,\xi)\in\Z\times\R | c_1k+c_2\xi-ic_3=0\}=\emptyset.$$
\end{corollary}
\begin{proof}
    Indeed, if $c_i\neq 0$, $i=1$ or $i=2$, then $L$ is SGH if and only if $\frac{1}{c_i}L$ is SGH, and we can apply the previous theorem to $\frac{1}{c_i}L$. 
\end{proof}

\begin{corollary}\label{coro1}
    Let $L$ be a differential operator on $\T^1\times\R$ given by
    $$L = \partial_t+(a+ib)\partial_x+q,$$
    where $a,b\in\R,\, q\in\mathbb{C}$. Then $L$ is globally hypoelliptic if and only if:
\begin{itemize}
    \item $b\neq 0$ and $a\frac{\Real(q)}{b}+\Imag(q)\not\in\Z$ or
    \item $b=0$ and $\Real(q)\neq0$ or
    \item  $b=\Real(q)=0=a$ and $\Imag(q)\not\in\Z$.
\end{itemize}
Also let $\Tilde{L}$ be a differential operator on $\T^1\times\R$ given by
    $$\Tilde{L} = (a+ib)\partial_t+\partial_x+q,$$
    where $a,b\in\R,\,q\in\mathbb{C}$. Then $\Tilde{L}$ is globally hypoelliptic if and only if:
    \begin{itemize}
    \item $b\neq 0$ and $\frac{\Real(q)}{b}\not\in\Z$ or
    \item $b=0$ and $\Real(q)\neq0$.
\end{itemize}
\end{corollary}
Note that in this last case, the global hypoellipticity of $\Tilde{L}$ does not depend on the value of $a$, nor $\Imag(q)$.\\
\begin{exemp}\label{exe1}
    Let $L_1$ and $L_2$ be the differential operators on $\T^1\times\R$ given by
    \begin{equation*}
        L_1=\partial_t+(a+bi)\partial_x+q,
    \end{equation*}
    and 
    \begin{equation*}
         L_2=(a+bi)\partial_t+\partial_x+q,
    \end{equation*}
    where $a,b,\Real(q),\Imag(q)\in\Z$, $b\neq0$. Then by Corollary \ref{coro1} we have that $L_1$ and $L_2$ are globally hypoelliptic if and only if $b$ does not divide $\Real(q)$.
\end{exemp}

\subsection{Variable Coefficients - Real case}

In this section, we prove that a differential operator given by
$$L = \partial_t+a(t)\partial_x+q(t),$$
where $a,q\in C^\infty(\T^1)$ and $a$ is real valued, is Schwartz globally hypoelliptic in $\T^1\times\R$ if and only in $L_0 = \partial_t+a_0\partial_x+q_0$ is Schwartz globally hypoelliptic, where 
$$a_0=\frac{1}{2\pi}\int_0^{2\pi} a(t)dt\quad\text{ and }\quad q_0=\frac{1}{2\pi}\int_0^{2\pi}q(t)dt.$$
The idea is to find an operator $\Psi:\mathcal{S}'(\T^1\times\R)\to\mathcal{S}'(\T^1\times\R)$ which preserves $\mathcal{S}(\g)$ and conjugates $L$ with $L_0$.\\
First, we deal with the function $a$.
\begin{prop}
    Let $a\in C^\infty(\T^1)$, $a_0\in\mathbb{R}$ be as before. Then the operator $\Psi_a:\mathcal{S}'(\T^1\times\R)\to\mathcal{S}'(\T^1\times\R)$ given by
    \begin{align*}
            (\Psi_au)(t,x) &= \frac{1}{2\pi}\int_{\R} e^{i\xi A(t)}\widehat{u}(t,\xi)e^{ix\xi}d\xi\\
            &= \mathcal{F}_{\R}^{-1}\left\{e^{i\xi A(t)}\widehat{u}(t,\xi)\right\},
    \end{align*}
    where 
    $$A(t) = \int_0^{t}a(s)ds-a_0t,$$
    is well defined and defines an automorphism over $\mathcal{S}(\T^1\times\R)$.
\end{prop}
\begin{proof}
    Using that $A\in C^\infty(\T^1)$ is real valued and the torus is compact, a simple calculation shows that $\Psi_a$ is well defined. Also, clearly $\Psi_a$ is invertible with $\Psi_a^{-1}$ given by
    \begin{align*}
        \Psi_a^{-1}u(t,x) &=\frac{1}{2\pi}\int_{\R} e^{-i\xi A(t)}\widehat{u}(t,\xi)e^{ix\xi}d\xi\\
            &= \mathcal{F}_{\R}^{-1}\left\{e^{-i\xi A(t)}\widehat{u}(t,\xi)\right\}. 
    \end{align*}
    Now let $u\in\mathcal{S}(\T^1\times\R)$, then for each $\gamma,\beta\in\N_0$ we have that
    \begin{align*}
        |\partial_\xi^\beta \widehat{\Psi_au}(t,\xi)|&=|\partial_{\xi}^\beta\left\{e^{i\xi A(t)}\widehat{u}(t,\xi)\right\}|\\
        &\leq \sum_{\beta'\leq \beta}C_{\beta'}|\xi^{\beta'}\partial_{\xi}^{\beta-\beta'}\widehat{u}(t,\xi)|,
    \end{align*}
    for some $C_{\beta'}>0$, so that by Proposition (\ref{propdecaypartialxi}) we have that $\Psi_au\in\mathcal{S}(\T^1\times\R)$. A similar argument works for $\Psi_a^{-1}$, so that $\Psi_a$ is an automporphism of $\mathcal{S}(\T^1\times\R)$.
\end{proof}

\begin{prop}
    Let $a$, $\Psi_a$ and $L$ be as before. Then we have that
    $$L_{a_0}\circ\Psi_{a} = \Psi_a\circ L_a,$$
    where 
    $$L_{a_0} = \partial_t+a_0\partial_x+q(t).$$
\end{prop}
\begin{proof}
    Note that, for $u\in\mathcal{S}'(\g)$, we have that
    $$L_{a_0}\circ\Psi_au = \partial_t\Psi_au+a_0\partial_x\Psi_au+q(t)\Psi_au.$$
    Taking the partial Fourier transform in $\R$ yields
    \begin{align}
    \partial_t\widehat{\Psi_au}(t,\xi)+i\xi &a_0\widehat{\Psi_au}(t,\xi)+q(t)\widehat{\Psi_au}(t,\xi) \notag \\
    =&\partial_t\left\{e^{i\xi A(t)}\widehat{u}(t,\xi)\right\}+i\xi a_0e^{i\xi A(t)}\widehat{u}(t,\xi)+q(t)e^{i\xi A(t)}\widehat{u}(t,\xi).\label{eqautom}
    \end{align}
    But as 
    \begin{align*}
        \partial_t\left\{e^{i\xi A(t)}\widehat{u}(t,\xi)\right\} = i\xi(a(t)-a_0)e^{i\xi A(t)}\widehat{u}(t,\xi)+e^{i\xi A(t)}\partial_t\widehat{u}(t,\xi),
    \end{align*}
     equation \eqref{eqautom} implies that
    \begin{align*}
        \widehat{L_{a_0}\circ\Psi_au}(t,\xi) &= e^{i\xi A(t)}(\partial_t\widehat{u}(t,\xi)+i\xi a(t)\widehat{u}(t,\xi)+q(t)\widehat{u}(t,\xi))\\
        &=\widehat{\Psi_a\circ L u}(t,\xi),
    \end{align*}
    from which the claim follows.
\end{proof}

\begin{prop}
    Let $q\in C^\infty(\T^1)$, $q_0\in\mathbb{C}$ be as before. Then the operator $\Psi_{q}:\mathcal{S}'(\g)\to\mathcal{S}'(\g)$ given by
$$(\Psi_qu)(t,x) = e^{Q(t)}{u}(t,x),$$
    where 
    $$Q(t) = \int_{0}^{t}q(s)ds-q_0t,$$
    is well defined and an automorphism of $\mathcal{S}(\g)$.
\end{prop}
\begin{proof}
    This follows from the definitions and compacity of the torus.
\end{proof}
The next proposition is a simple calculation, and its proof is left to the reader.
\begin{prop}
    Let $q$, $\Psi_q$, $L_{a_0}$ and $L_0$ be as before. We then have that
    $$L_0\circ \Psi_q = \Psi_q\circ L_{a_0}.$$
\end{prop}
Finally we obtain the necessary conjugation and equivalence, as follows.
\begin{corollary}\label{coro_real}
    Let $\Psi_a,\Psi_q$ be as before. Then 
    $$L_0\circ\Psi= L\circ \Psi$$
    and
    $$L \circ \Psi^{-1} = \Psi^{-1}\circ L_0,$$
    where $\Psi = \Psi_a\circ\Psi_q$. Therefore $L$ is Schwartz globally hypoelliptic if and only if $L_0$ is Schwartz globally hypoelliptic.
\end{corollary}
\begin{proof}
    The equalities above are an immediate consequence of the previous propositions. The proof of the second claim is as follows: note that if $L$ is SGH and $L_0u=f\in\mathcal{S}(\g)$, then $L\circ \Psi^{-1} u = \Psi^{-1}\circ L_0u=\Psi^{-1} f\in\mathcal{S}(\g)$ so that $\Psi^{-1} u\in\mathcal{S}(\g)\implies u\in\mathcal{S}(\g)$ and so $L_0$ is SGH. On the other hand, if $L_0$ is SGH and $Lu=f\in\mathcal{S}(\g)$, then $L_0\circ\Psi u = \Psi\circ L u=\Psi f\in\mathcal{S}(\g)$, so that $\Psi u\in\mathcal{S}(\g)\implies u\in\mathcal{S}(\g)$ and so $L$ is SGH.
\end{proof}
\begin{exemp}
    Let $L_1$ be the differential operator on $\T^1\times\R$ given by
    \begin{equation*}
        L_1=\partial_t+a(\sin(t)+1)\partial_x+q,
    \end{equation*}
    where $a\in\R$, $q\in\mathbb{C}$. Then by Corollary \ref{coro_real} and Theorem \ref{coro_cte_cases}, the operator $L_1$ is SGH if and only if $\Real(q)\neq 0$, or $\Real(q)=a=0$ and $\Imag(q)\not\in\Z$.
\end{exemp}
\bibliographystyle{plain} 
\bibliography{references} 
\nocite{*}
\end{document}